\documentclass[12pt]{amsart}
\usepackage[utf8]{inputenc}

\usepackage{amsmath,amsthm,amssymb,mathscinet,bbm}
\usepackage{parskip}
\usepackage{geometry}
\usepackage{xcolor} 

\usepackage{mathtools} 
\usepackage[hidelinks]{hyperref}
\author{Paul Pollack} 
\address{Department of Mathematics \\ University of Georgia \\ Athens, GA 30602}
\email{pollack@uga.edu}

\subjclass[2020]{Primary 11N64; Secondary 11A25, 11B50}

\renewcommand\phi\varphi
\usepackage{accents}

\renewcommand{\pod}[1]{\allowbreak\mathchoice
  {\if@display \mkern 18mu\else \mkern 8mu\fi (#1)}
  {\if@display \mkern 18mu\else \mkern 8mu\fi (#1)}
  {\mkern4mu(#1)}
  {\mkern4mu(#1)}
}
\usepackage{graphicx}
\DeclareMathAlphabet{\curly}{U}{rsfs}{m}{n}
\newcommand{\odd}{\mathrm{odd}}

\newcommand{\Li}{\mathrm{Li}}
\newcommand\Z{\mathbb{Z}}
\newcommand\lcm{\mathrm{lcm}}

\renewcommand\v{\mathbf{v}}
\newcommand\V{\mathcal{V}}

\newtheorem{thm}{Theorem}[section]

\newtheorem{prop}[thm]{Proposition}
\newtheorem{lem}[thm]{Lemma}

\theoremstyle{remark}

\numberwithin{equation}{section}
\begin{document}
\title{Two problems on the distribution of Carmichael's lambda function}
\begin{abstract} Let $\lambda(n)$ denote the exponent of the multiplicative group modulo $n$. We show that when $q$ is odd, each coprime residue class modulo $q$ is hit equally often by $\lambda(n)$ as $n$ varies. Under the stronger assumption that $\gcd(q,6)=1$, we prove that equidistribution persists throughout a Siegel--Walfisz-type range of uniformity. By similar methods we show that $\lambda(n)$ obeys Benford's leading digit law with respect to natural density. Moreover, if we assume GRH, then Benford's law holds for the order of $a$ mod $n$, for any fixed integer $a\notin \{0,\pm 1\}$.\end{abstract}
\maketitle
 
\section{Introduction} If $f$ is a naturally-occurring integer-valued arithmetic function, it is reasonable to ask how the values of $f$ are distributed in arithmetic progressions. Several results of this nature are collected in Narkiewicz's monograph \cite{narkiewicz84}. In particular, Chapter IV of that reference describes in detail an easily-applicable criterion of Delange (appearing originally as \cite[Theorem 1]{delange69}) deciding, for each additive function $f$ and each modulus $q$, whether $f$ is uniformly distributed (UD) modulo $q$. 

The situation for multiplicative functions is more subtle. Here Euler's $\phi$-function is a convenient reference point. It has long been known that for each fixed positive integer $q$, the relation $q \mid \phi(n)$ holds for almost all $n$, meaning for all but $o(x)$ values of $n\le x$ as $x\to\infty$. (This, and much more, follows from arguments given below.) That is, the class of $0$ modulo $q$ `hogs' almost all the values of $\phi(n)$, ruling out uniform distribution mod $q$ except in the trivial case $q=1$. This phenomenon prompted Narkiewicz in \cite{narkiewicz67} to introduce a weaker notion of equidistribution: \textsf{$f$ is weakly uniformly distributed mod $q$} (or \textsf{WUD} mod $q$) if (a) $\gcd(f(n),q)=1$ for infinitely many $n$, and (b) for each coprime residue class $a\bmod{q}$,
\begin{equation}\label{eq:WUDdef} \#\{n\le x: f(n)\equiv a\pmod{q}\} \sim \frac{1}{\phi(q)}\#\{n\le x: \gcd(f(n),q)=1\}, \quad\text{as~$x\to\infty$}. \end{equation}
In the same paper, Narkiewicz gives a criterion for weak uniform distribution that may be applied to multiplicative functions that are `polynomial-like' (in a sense we will not describe here). As applications, he classifies the moduli $q$ for which $d(n)$ (the divisor function), and $\phi(n)$, are WUD mod $q$. For instance, $\phi(n)$ is WUD mod $q$ precisely when $\gcd(q,6)=1$. For a full development of the theory of weak uniform distribution of polynomial-like multiplicative functions (incorporating later refinements by Narkiewicz and collaborators), see Chapters V and VI of the previously mentioned monograph \cite{narkiewicz84}.

The central object of study in this paper is Carmichael's $\lambda$-function \cite{carmichael19}, which is a close cognate of Euler's function. Whereas $\phi(n)$ gives the \emph{order} of the unit group $U(\Z/n\Z)$, Carmichael's function $\lambda(n)$ describes its \emph{exponent}; that is, $\lambda(n)$ is the smallest positive integer for which $a^{\lambda(n)}\equiv 1\pmod{n}$ for each $a$ coprime to $n$. The $\lambda$ function is not multiplicative but is what might be called \textsf{lcm-multiplicative}: 
\begin{equation}\label{eq:lambdalcmrelation} \lambda(\lcm[m,n]) = \lcm[\lambda(m),\lambda(n)], \quad\text{for all positive integers $m,n$.} \end{equation}
As known already to Gauss, $\lambda(p^k) = \phi(p^k)$ whenever $p^k$ is an odd prime power, $\lambda(1) = \lambda(2) = 1$, $\lambda(4) = 2$, and $\lambda(2^k) = 2^{k-2}$ for all $k\ge 3$; these values, along with the relation \eqref{eq:lambdalcmrelation}, determine $\lambda(n)$ for all inputs $n$.

Several statistical properties of $\lambda$, such as its average, typical, and lower order, are investigated by Erd\H{o}s, Pomerance, and Schmutz in \cite{EPS91}. One takeaway from their study is that despite having similar definitions, $\phi$ and $\lambda$ can behave quite differently. To give just one example: It is classical that $\phi(n) \gg n/\log\log{(3n)}$ for all $n$, whereas it is shown in \cite{EPS91} that there is a sequence of $n$ tending to infinity along which $\lambda(n) \le (\log{n})^{O(\log\log\log{n})}$. As far as we are aware, there have not been prior investigations into the distribution of $\lambda(n)$ in residue classes. It is easy to prove that $\lambda(n)$ is even whenever $n\ge 3$, and so $\lambda(n)$ cannot be WUD mod $q$ if $2\mid q$. In our first theorem we prove that this is the only obstruction.

\begin{thm}\label{thm:WUD1} $\lambda$ is WUD mod $q$ for all odd $q$. In fact, as $x\to\infty$,
\begin{equation}
     \#\{n\le x: \lambda(n) \equiv a\pmod{q}\} \sim \frac{1}{\phi(q)} \#\{n\le x: \gcd(\lambda(n),q)=1\}, \label{eq:lambdaWUD}
\end{equation}
uniformly for coprime residue classes $a\bmod{q}$ with $q$ odd and $q\le \log\log\log{x}$.
\end{thm}

The range of uniformity in Theorem \ref{thm:WUD1} is rather modest (to put it mildly). If we assume that $3\nmid q$, we can do much better.

\begin{thm}\label{thm:WUD2} Fix $A>0$. As $x\to\infty$, the relation \eqref{eq:lambdaWUD} holds uniformly for moduli $q \le (\log{x})^{A}$ with $\gcd(q,6)=1$.
\end{thm}

Our next two theorems are of a slightly different nature. To set the stage, fix an integer $b\ge 2$.  Let $D$ be a positive integer, and let $x$ be a positive real number. We say that \textsf{$x$ begins with $D$ in base $b$} if the most significant digits of $x$ in base $b$ are the base $b$ digits of $D$.  For example, $357$ and $0.03512$ both begin with $D=35$ in base $10$. A sequence $\{a_n\}$ of positive real numbers is said to obey \textsf{Benford's law} in base $b$ if, for every positive integer $D$, the asymptotic density of $n$ for which $a_n$ begins with $D$ is $\log(1+D^{-1})/\log{b}$. 

Diaconis observed in \cite{diaconis77} that $\{a_n\}$ is Benford in base $b$ precisely when the sequence $\{\frac{\log a_n}{\log{b}}\}$ is uniformly distributed modulo $1$. It is useful to rephrase this conclusion  using Weyl's criterion. For each integer $k$, let $\theta_k = 2\pi k/\log{b}$. Then $\{a_n\}$ is Benford in base $b$ precisely when $a_n^{i\theta_k}$ has limiting mean value $0$, for each $k=1,2,3,\dots$. This criterion was recently used in \cite{CLPSR23} to study Benford behavior of sequences described by multiplicative functions. For example, it was shown there that the sequences $\{\phi(n)\}$ and $\{\sigma(n)\}$ (with $\sigma$ the sum-of-divisors function) are not Benford in any base $b$.\footnote{Here it is important that we use asymptotic density in our definition. If we were to use logarithmic density instead, $\{\phi(n)\}$ and $\{\sigma(n)\}$ could be shown to obey Benford's law in every base.} On the other hand, the sequence $\{\tau_k(n)\}$ (with $\tau_k$ the $k$-fold divisor function) is Benford in base $b$ if and only if $\log{k}/\log{b}$ is irrational. (Actually the work in \cite{CLPSR23} is carried out in base $10$, but the arguments generalize to arbitrary bases.) 

Our next theorem implies that $\{\lambda(n)\}$ is Benford in every base.  

\begin{thm}\label{thm:Benford} Fix a nonzero real number $\theta$. Then $\sum_{n\le x} \lambda(n)^{i\theta} = o(x)$, as $x\to\infty$.
\end{thm}

Let $a$ be an integer with $|a|>1$. For each positive integer $n$ coprime to $a$, let $\ell_a(n)$ denote the multiplicative order of $a$ mod $n$. In our final theorem we show that, under the assumption of the Generalized Riemann Hypothesis\footnote{here we mean the Riemann Hypothesis for Dedekind zeta functions} (GRH), we can replace $\lambda(n)$ with $\ell_a(n)$ in this last result.

\begin{thm}[conditional on GRH]\label{thm:Benford2} Fix an integer $a$ with $|a| > 1$. For each fixed nonzero real number $\theta$, 
\[ \sum_{\substack{n\le x \\ \gcd(n,a)=1}} \ell_a(n)^{i\theta} = o(x), \qquad\text{as $x\to\infty$}. \]
\end{thm}

Thus, Benford's law holds for the orders $\ell_a(n)$ (where now the relevant densities are to be computed relative to the set of $n$ with $\gcd(n,a)=1$).

We conclude this introduction with a word about the proofs. The UD and WUD criteria of Delange and Narkiewicz alluded to in the introduction are proved by recasting the respective problems in terms of mean values of multiplicative functions of modulus not exceeding $1$. (For additive functions, this involves composing with an additive character, while for multiplicative functions one uses Dirichlet characters.) The authors of \cite{CLPSR23} adopt a similar perspective in their work on Benford's law (note that if $f$ is positive-valued and multiplicative, then $f^{i \theta}$ is a multiplicative function of modulus $1$, for every real $\theta$). Such an approach allows one to bring to bear powerful tools such as Hal\'asz's theorem. 

Since $\lambda$ is not multiplicative, we must take a different tack. In recent work with Singha Roy \cite{PSR23}, we proposed an alternative method of proving UD and WUD theorems. This was used to show (among other things) that for $f=\phi$, the relation \eqref{eq:WUDdef} holds uniformly for $q\le (\log{x})^{A}$ with $\gcd(q,6)=1$. Theorem \ref{thm:WUD2} is proved in \S\ref{sec:WUD2} using these same ideas. Here it is important that the contribution of large primes to $\lambda(n)$ can usually be computed as if $\lambda$ were multiplicative; see the discussion following \eqref{eq:coprimecondition}.

The arguments for Theorem \ref{thm:WUD2} fail to establish weak uniform distribution when $3 \mid q$ but do show (as described in \S\ref{sec:interlude}) that a failure of \eqref{eq:lambdaWUD} in this case entails the nonuniform distribution of $\lambda(n)$ mod $3$ (among $n$ with $\gcd(\lambda(n),q)=1$). In \S\ref{sec:thm1proof}, we rule out this pathology in the range $q\le \log\log\log{x}$, thus proving Theorem \ref{thm:WUD1}. This requires bringing in certain anatomical facts of the kind that frequently arise when studying Euler's $\phi$-function. For instance, it may be instructive to compare the proof of our Lemma \ref{lem:gcdformula} with that of Theorem 8 in \cite{ELP08}.

Theorems \ref{thm:Benford} and \ref{thm:Benford2} are proved in \S\ref{sec:benford}. The proof of Theorem \ref{thm:Benford} is a fairly straightforward adaptation of the arguments in \S\ref{sec:thm1proof}. The proof of Theorem \ref{thm:Benford2} has the same basic structure but also requires results on the distribution of the numbers $\ell_a(p)$, which we extract from work of Li--Pomerance \cite{LP03}, Moree \cite{moree05}, and Pappalardi \cite{pappalardi15}.

\subsection*{Notation and conventions} We use $(a,b)$ to denote the greatest common divisor of $a$ and $b$. The letters $\ell, p$, and $P$ should always be read as being restricted to primes values. We write $\log_k$ for the $k$-fold iterate of the natural logarithm. Implied constants are usually absolute but are allowed depend on parameters described explicitly as `fixed'; concretely, this means that constants appearing in the proof of Theorem \ref{thm:WUD2} may depend on $A$, that those appearing in the proof of Theorem \ref{thm:Benford} may depend on $\theta$, and that those appearing in the proof of Theorem \ref{thm:Benford2} may depend on both $\theta$ and $a$.

We write $P^{+}(n)$ for the largest prime factor of the positive integer $n$ and adopt the convention that $P^{+}(1)=1$. If $P^{+}(n) \le y$, we say that $n$ is \textsf{$y$-smooth}. The \textsf{$y$-smooth part} of $n$ refers to the largest $y$-smooth divisor of $n$; we denote this by $s_y(n)$, so that
\[ s_y(n) = \prod_{\substack{p^k \parallel n\\ p\le y}} p^k. \]
We let $P_1(n)=P^{+}(n)$ and define, inductively, $P_{k+1}(n) = P^{+}(n/P_{k}(n))$; thus, $P_k(n)$ is the $k$th largest prime factor of $n$ with multiplicity taken into account. 

We will also need the \textsf{Schemmel totient} function $\phi_2(n)$, defined as the count of residue classes $a\bmod{n}$ with $\gcd(a(a-1),n)=1$. Familiar arguments show that $\phi_2(n) = n\prod_{p\mid n} (1-2/p)$.

\section{Preliminaries}
In this section we record some results needed for the proofs of Theorems \ref{thm:WUD1}--\ref{thm:Benford}. They are all closely related to lemmas appearing in \cite{PSR23}.

For each positive integer $q$, we let $\alpha(q) = \prod_{p \mid q} \left(1-1/(p-1)\right)$. It will be important for the sequel that when $q$ is odd,
\begin{equation}\label{eq:alphalower} \alpha(q) \gg 1/\log_2{(3q)}. \end{equation}
We will very often abbreviate $\alpha(q)$ to $\alpha$. 

The following proposition is a special case of \cite[Proposition 2.1]{PSR23}. A more precise estimate could be extracted from work of Scourfield \cite{scourfield84}, but we shall not need that.

\begin{prop}\label{prop:sumcoprime} Fix $A>0$. For $x$ tending to infinity and all odd integers $q \le (\log{x})^A$, 
  \begin{equation}\label{eq:sumcoprime} \#\{n\le x: (\lambda(n),q)=1\} = \frac{x}{(\log{x})^{1-\alpha}} \exp(O((\log_2(3q))^{2})).  \end{equation}
  \end{prop} 

Actually, Proposition 2.1 of of \cite{PSR23} estimates the frequency with which a general ``polynomial-like'' multiplicative function $f$ is coprime to a given integer $q$. Taking $f=\phi$ gives Proposition  \ref{prop:sumcoprime}, after observing that the conditions $(\lambda(n),q)=1$ and $(\phi(n),q)=1$ are equivalent. (For any finite abelian group, the exponent and order share the same set of prime factors.) We note that in \cite{PSR23}, the exponent of $\log_2{(3q)}$ is given as $O(1)$, but inspecting the proof reveals that this exponent can be taken as $2$ for $f=\phi$.

Proposition \ref{prop:sumcoprime} is established by a combination of sieve methods and mean value theorems. The key arithmetic input is an estimate for $\sum_{p\le x,~(p-1,q)=1} 1/p$ (see \cite[Lemma 2.4]{PSR23}). We now restate this estimate alongside a slight extension that will be needed later.
\begin{prop}\label{prop:slightgen} For each positive integer $q$, and all $x\ge 3$, 
  \[ \sum_{\substack{p \le x \\ (p-1,q)=1}} \frac{1}{p} = \alpha(q) \log_2{x} + O((\log_2{(3q)})^2). \]
If additionally $s$ is a positive integer coprime to $q$, then
  \[ \sum_{\substack{p \le x \\ (p-1,q)=1 \\ s\nmid p-1}} \frac{1}{p} = \alpha(q)\left(1-\frac{1}{\phi(s)}\right)\log_2{x} + O((\log_2{(3q)})^2). \]
\end{prop}

For completeness, we sketch the proof of Proposition \ref{prop:slightgen}. We need a lemma due independently to Norton \cite[Lemma, p.\ 669]{norton76} and Pomerance \cite[Remark 1]{pomerance77}.

\begin{lem}\label{lem:PN} Let $q$ be a positive integer and let $x$ be a real number with $x \ge \max\{3,q\}$. For each coprime residue class $a\bmod{q}$,
\[ \sum_{\substack{p\le x \\ p\equiv a\pmod{q}}} \frac{1}{p} = \frac{\log_2{x}}{\phi(q)} + \frac{1}{p_{q,a}} + O\left(\frac{\log(3q)}{\phi(q)}\right), \]
where $p_{q,a}$ denotes the least prime congruent to $a$ modulo $q$.
\end{lem}

In particular: For all positive integers $q$ and all $x\ge 3$, we have $\sum_{p \le x,~p\equiv 1\pmod{q}} 1/p = \log_2{x}/\phi(q) + O(\log{(3q)}/\phi(q))$ (this estimate being trivial when $q > x$).

\begin{proof}[Proof of Proposition \ref{prop:slightgen} {\rm (}sketch{\rm )}] We write
\[ \sum_{\substack{p \le x \\ (p-1,q)=1 \\ s\nmid p-1}} \frac{1}{p} = \sum_{\substack{p \le x \\ (p-1,q)=1}} \frac{1}{p} - \sum_{\substack{p \le x \\ (p-1,q)=1 \\ s\mid p-1}} \frac{1}{p} \]
and proceed to estimate the two right-hand sums. By inclusion-exclusion and Lemma \ref{lem:PN},
\[ \sum_{\substack{p \le x \\ (p-1,q)=1}} \frac{1}{p} = \sum_{d \mid q} \mu(d) \sum_{\substack{p \le x \\ p\equiv 1\pmod{d}}} \frac{1}{p} = \log_2{x} \sum_{d\mid q} \frac{\mu(d)}{\phi(d)} + O\left(\sum_{d \mid q}|\mu(d)| \frac{\log(3d)}{\phi(d)}\right). \]
Similarly, 
\[ \sum_{\substack{p \le x \\ (p-1,q)=1 \\ s\mid p-1}} \frac{1}{p} = \sum_{d \mid q} \mu(d) \sum_{\substack{p \le x \\ p\equiv 1\pmod{sd}}} \frac{1}{p} = \frac{\log_2{x}}{\phi(s)} \sum_{d\mid q} \frac{\mu(d)}{\phi(d)} + O\left(\sum_{d \mid q} |\mu(d)| \frac{\log(3ds)}{\phi(sd)}\right). \]
Since $\sum_{d\mid q} \mu(d)/\phi(d) = \alpha(q)$, we have the main terms claimed in the proposition. As far as the errors, notice that $\log(3ds)/\phi(sd) \le (\log{(3d)}/\phi(d)) \cdot (\log(3s)/\phi(s)) \ll \log(3d)/\phi(d)$. Thus, the proof will be completed once it is shown that $\sum_{d \mid q} |\mu(d)| \log(3d)/\phi(d) \ll (\log_2(3q))^2$. But this elementary estimate is already worked out (in greater generality) at the end of the proof of \cite[Lemma 2.4]{PSR23}. (In the notation of that argument, we have $F(T) = T-1$, $D=1$, and $\nu(d)=1$ for all $d$.) 
\end{proof}

The last result we need can be found, in slightly different form, near the start of \cite[\S4]{PSR23}. For convenience of the reader, we include a simple proof (a variant of the ``Remark'' in \cite{PSR23}).

\begin{lem}\label{lem:congruencecount} Let $q$ be a positive integer coprime to $6$. For each integer $J\ge 2$ and each integer $r$ prime to $q$, 
\begin{multline*} \#\{(a_1,\dots,a_J)\bmod{q}: \gcd\bigg(\prod_{j=1}^{J} a_j(a_j-1),q\bigg)=1,~\prod_{j=1}^{J} (a_j-1) \equiv r\pmod{q}\} \\
= \frac{\phi_2(q)^{J}}{\phi(q)} \exp\left(O\bigg(\sum_{p\mid q} (p-2)^{1-J}\bigg)\right).\end{multline*} 
\end{lem}
\begin{proof} By the Chinese remainder theorem, we can assume $q=p^k$ is a prime power with $p\ge 5$. Orthogonality of Dirichlet characters implies that the cardinality in question is 
\[ \frac{1}{\phi(p^k)} \sum_{\chi \bmod{p^k}} \left(\sum_{a\bmod p^k}\chi_0(a) \chi(a-1)\right)^J \bar{\chi}(r). \]
Here the principal character $\chi=\chi_0$ contributes $\phi_2(p^k)^{J}/\phi(p^k)$. For $\chi \ne \chi_0$, the complete sum $\sum_{a\bmod p^k} \chi(a-1)=0$, so that
\[ \sum_{a\bmod p^k}\chi_0(a) \chi(a-1) = -\sum_{\substack{a\bmod p^k \\ a\equiv 0\pmod{p}}} \chi(a-1)= -\chi(-1) \sum_{\substack{b\bmod p^k \\ b\equiv 1\pmod{p}}} \chi(b). \]
The residue classes $b\bmod{p^k}$ with $b\equiv 1\pmod{p}$ form an index $p-1$ subgroup of the multiplicative group mod $p^k$. Thus, the sum on $b$ vanishes unless the restriction of $\chi$ to this subgroup is trivial. This happens for precisely $(p-1)-1=p-2$ nontrivial characters, and in these cases $\sum_{a\bmod{p^k}} \chi_0(a)\chi(a-1)$ has absolute value $p^{k-1}$. The triangle inequality thus implies that 
\[ \frac{1}{\phi(p^k)} \left|\sum_{\substack{\chi \bmod{p^k} \\ \chi\ne \chi_0}} \left(\sum_{a\bmod p^k}\chi_0(a) \chi(a-1)\right)^J \bar{\chi}(r)\right| \le \frac{1}{\phi(p^k)} (p-2) p^{(k-1)J}. \] Putting this estimate back above, our tuple count is $\frac{\phi_2(p^k)^J}{\phi(p^k)}(1+\theta (p-2)^{1-J})$, for some real number $\theta$ with $|\theta| < 1$. As $1+\theta (p-2)^{1-J} = \exp(O((p-2)^{1-J}))$, the lemma follows.
\end{proof}

In our application of Lemma \ref{lem:congruencecount} we will have that $J\to\infty$. Thus, $\sum_{p \mid q} (p-2)^{1-J} = o(1)$, leading to a tuple count that is $\sim \phi_2(q)^J/\phi(q)$.

\section{Uniformity up to $(\log{x})^{A}$ when $(q,6)=1$: Proof of Theorem \ref{thm:WUD2}}\label{sec:WUD2}

In what follows, $x$ is assumed to be a large real number, $q$ is an odd positive integer with $q\le (\log{x})^{A}$, where $A>0$ is fixed, and $a\bmod{q}$ is a coprime residue class. Unless otherwise stated, asymptotic estimates refer to behavior as $x\to\infty$ and are to be read as uniform in the choice of $a\bmod{q}$. Implied constants may depend on $A$ but on no other parameters unless explicitly noted. 
 
We begin the proof of Theorem \ref{thm:WUD1}  by restricting our attention to inputs $n$ whose large prime factors are suitably well-behaved. We let
\[ J:= \lfloor \log\log\log{x}\rfloor.\]
(Any integer-valued function of $x$ tending to infinity sufficiently slowly would do just as well.) We also set
\[ y: = \exp(\sqrt{\log{x}}). \] Following \cite[\S3]{PSR23} (choosing $\delta:=1$ in the notation of that paper), a positive integer $n\le x$ is called \textsf{convenient} if $P_J(n) > y$ and none of the primes $P_1(n), P_{2}(n), \dots, P_J(n)$ have squares dividing $n$. That is, $n$ is convenient when $n$ admits a decomposition
\begin{equation}\label{eq:convenientdecomposition} n = mP_J \cdots P_1, \quad\text{with}\quad L_m:=\max\{P^{+}(m),y\} < P_J < \cdots < P_1. \end{equation}

Write $N = N(x,q)$ for the total count of $n\le x$ with $(\lambda(n),q)=1$. The following lemma, proved as Lemma 3.1 in \cite{PSR23}, shows that this count does not change very much if we restrict to convenient $n$. 

\begin{lem}\label{lem:inconvenient0} The number of inconvenient $n\le x$ with $(\lambda(n),q)=1$ is $o(N)$. 
\end{lem}

Our next lemma asserts that there are few inconvenient solutions to $\lambda(n)\equiv a \pmod{q}$.

\begin{lem}\label{lem:inconvenientsolutions} The number of inconvenient $n\le x$ with $\lambda(n)\equiv a\pmod{q}$ is $o(N/\phi(q))$.
\end{lem}

For the proof of Lemma \ref{lem:inconvenientsolutions} (and subsequently), it is useful to note that \eqref{eq:alphalower} and \eqref{eq:sumcoprime}  imply (crudely) that $N \gg x/\log{x}$ in our range of $q$. 

In addition, it will be helpful to have at hand certain basic estimates from the theory of smooth numbers. Let $\Psi(X,Y)$ denote the count of $Y$-smooth numbers in $[1,X]$. Canfield, Erd\H{o}s, and Pomerance have shown \cite{CEP83} that if $X,Y$, and $U:=\frac{\log{X}}{\log{Y}}$ all tend to infinity, with $X\ge Y\ge (\log{X})^2$, then
\begin{equation}\label{eq:CEP} \Psi(X,Y)= X \exp(-(1+o(1)) U\log{U}). \end{equation}
Also, it is known (see \cite[Theorem 5.1, p.\ 512]{tenenbaum15}) that for all $X\ge Y \ge 2$,
\begin{equation}\label{eq:cleanbound} \Psi(X,Y) \ll X \exp(-U/2).\end{equation}

\begin{proof}[Proof of Lemma \ref{lem:inconvenientsolutions}] We may assume that
\begin{enumerate}
  \item[(i)] $n$ has no repeated prime factor exceeding $(\log{x})^{A+1}$,
\end{enumerate}
since the  number of $n\le x$ failing (i) is $o(x/(\log{x})^{A+1})$, which is $o(N/\phi(q))$ in our range of $q$. We may also assume that, with $y':=x^{1/\log_2{x}}$,
\begin{enumerate}
   \item[(ii)] $P^{+}(n) > y'$.
\end{enumerate}
Indeed, \eqref{eq:CEP} shows that (ii) holds for all but $x/(\log{x})^{(1+o(1))\log_3 x} = o(N/\phi(q))$ values of $n\le x$. We write each remaining inconvenient solution $n\le x$ to $\lambda(n)\equiv a\pmod{q}$ in the form $n=rP$, where $P=P^{+}(n)$. By (i) and (ii), we know that $P\nmid r$, and so 
\begin{equation}\label{eq:lambdaeq0} \lambda(n) = \lcm[\lambda(r),P-1] = \frac{\lambda(r) \cdot (P-1)}{(\lambda(r),P-1)}. \end{equation}

Continuing, let $y'' = \exp((\log_2{x})^2)$. We may further suppose that
\begin{enumerate}
  \item[(iii)] $(\lambda(r),P-1)$ is $y''$-smooth.
\end{enumerate}
Indeed, if (iii) fails choose a prime $\ell >y''$ for which $\ell$ divides $(\lambda(r),P-1)$.  By (i), $\ell^2 \nmid r$, and so there must be a prime $p$ dividing $r$ with $p\equiv 1\pmod{\ell}$. There are  at most $x/pP$ values of $n$ corresponding to a given pair of $p, P$. Summing $x/pP$ on \emph{all} primes $p, P\equiv 1\pmod{\ell}$ with $p,P\le x$ bounds the count of $n$ corresponding to a given $\ell$ by $x (\log_2{x})^2/\ell^2$. (We have bounded the sums on $p,P$ by partial summation and the Brun--Titchmarsh inequality.) Now summing on $\ell > y''$ shows that the number of exceptions to (iii) is $O(x(\log_2{x})^2/y'')$, which is $o(N/\phi(q))$. 

Continuing, we may suppose that
\begin{enumerate}
  \item[(iv)] $(\lambda(r),P-1) \le y'^{1/3}$.
\end{enumerate}
When (iv) fails, we use (iii) to decompose $P-1 = dh$, where $d$ is $y''$-smooth and $d > y'^{1/3}$. Since the number of possibilities for $n$ given $P$ is at most $x/P < x/dh$, summing on $h$ and $d$ bounds the number of exceptions to (iv) as 
\[ \ll x \log{x} \sum_{\substack{d > y'^{1/3} \\ d \text{ $y''$-smooth}}} \frac{1}{d}. \]
Summing by parts, \eqref{eq:cleanbound} gives that the sum on $d$ is $O(\log{y''} \cdot \exp(-\frac{1}{6}\frac{\log{y'}}{\log{y''}}))$. Thus, the number of $n$ where (iv) fails is $O(x \exp(-(\log{x})^{0.9}))$ (say), which is $o(N/\phi(q))$.

We now fix $d:=(\lambda(r),P-1)$ and count possibilities for $P$ given $r$ and $d$. We must have $(\lambda(r),q)=1$ (otherwise $(\lambda(n),q) > 1$, contradicting $\lambda(n)\equiv a\pmod{q}$) and hence also $(d,q)=1$. Furthermore, \[ P\equiv 1\pmod{d} \] and, from \eqref{eq:lambdaeq0}, 
\[ P-1 \equiv \frac{ad}{\lambda(r)} \pmod{q}. \]
(The right-hand side of the congruence makes sense mod $q$, since $(\lambda(r),q)=1$.) These two congruences put $P$ in a unique residue class mod $dq$. Since $y' < P \le x/r$ while $d, q < y'^{1/3}$, the Brun--Titchmarsh inequality gives that the number of possibilities for $P$ given $r, d$ is 
\[ \ll \frac{x}{\phi(dq) r \log(x/rdq)} \ll \frac{x \log_2{x}}{\phi(d) \phi(q) r \log{x}} \ll  \frac{x (\log_2{x})^2}{d \phi(q) r \log{x}},
\]
using in the last step that $d/\phi(d) \ll \log_{2}{d} \ll \log_{2}x$. Since $d$ is $y''$-smooth, it follows that the number of $P$ given $r$ is 
\begin{equation}\label{eq:Pgivenr} \ll \frac{x (\log_2{x})^2}{\phi(q) r \log{x}} \prod_{p \le y''} \left(1+\frac1p + \frac1{p^2} + \dots\right) \ll \frac{x (\log_2{x})^4}{\phi(q) r \log{x}}.
\end{equation}
Finally we sum on $r$. Write $r=r_1 r_2$, where $r_1$ is the $y$-smooth part of $r$. We bound $\sum 1/r$ by summing on all possibilities for $r_1$ and $r_2$. Since $n=r_1 r_2 P$ is inconvenient, it must be that $\Omega(r_2) < J$. Thus,
\begin{align*} \sum_{r} \frac{1}{r} &\le \left(1+\sum_{p \le x} (1/p + 1/p^2 + 1/p^3 +\dots)\right)^{J} \sum_{r_1} \frac{1}{r_1} \\
  &\le \exp(O((\log_3{x})^2)) \cdot \sum_{r_1} \frac{1}{r_1}. 
\end{align*}
Since $r_1$ is divisible only by primes $p \le y$ with $(p-1,q)=1$,
\begin{multline*} \sum_{r_1} \frac{1}{r_1} \le \prod_{\substack{p \le y \\ (p-1,q)=1}} \left(\sum_{j \ge 0}p^{-j}\right) \ll \exp\Bigg(\sum_{\substack{p \le y \\ (p-1,q)=1}}\frac{1}{p}\Bigg) \ll (\log{x})^{\frac{1}{2}\alpha} \exp(O((\log_2{(3q)})^2)),
\end{multline*}
using Proposition \ref{prop:slightgen} to estimate the final sum on $p$. Plugging these estimates back into \eqref{eq:Pgivenr}, we conclude that the number of inconvenient $n\le x$ with $\lambda(n)\equiv a\pmod{q}$ and all of (i)--(iv) being satisfied is 
\[ \ll \frac{x}{\phi(q)(\log{x})^{1-\frac{1}{2}\alpha}}  \exp(O((\log_2{(3q)})^2 + (\log_3{x})^2)). \]
Comparing with the estimate for $N$ from Proposition \ref{prop:sumcoprime}, and keeping in mind the lower bound \eqref{eq:alphalower} on $\alpha$, this is seen to be $o(N/\phi(q))$ (in fact, it is $O(N/\phi(q) \exp((\log_2{x})^{0.9}))$, say). 
\end{proof}

By Lemmas \ref{lem:inconvenient0} and \ref{lem:inconvenientsolutions}, Theorem \ref{thm:WUD2} will follow if it is shown that
\begin{multline}\label{eq:convenientreduced} \#\{\text{convenient $n\le x$}: \lambda(n)\equiv a\pmod{q}\} \\ = \frac{1}{\phi(q)}\#\{\text{convenient $n\le x$}: (\lambda(n),q)=1\} + o(N/\phi(q)). \end{multline}
Referring back to the description of convenient numbers appearing as \eqref{eq:convenientdecomposition}, the left-hand side of \eqref{eq:convenientreduced} can be rewritten as
\begin{equation}\label{eq:hugedoublesum}  \sum_{\substack{m\le x \\ (\lambda(m),q)=1}} \sum_{\substack{P_j, \dots, P_1 \\ L_m < P_J < P_{J-1} < \dots < P_1\\ P_J \cdots P_1  \le x/m \\ \lcm[\lambda(m),P_1-1,\dots,P_J-1]\equiv a\pmod{q}}} 1. \end{equation}

Proceeding futher requires that we tame the ornery-seeming lcm condition. To this end, we let
\[ w:= \exp((\log{x})^{1/8}) \]
and we show that most of the time 
\begin{equation}\label{eq:coprimecondition} \text{$\lambda(m), P_1-1, \dots, P_j-1$\quad are  pairwise coprime at primes $>w$,} \end{equation} meaning that no prime $\ell >w$ divides two terms from the list. Suppose there is such an $\ell$. Then either $n = mP_1\cdots P_J$ is divisible by $\ell^2$ or there are primes $p_1, p_2$ dividing $n$ with $p_1, p_2\equiv 1\pmod{\ell}$. The number of $n\le x$ satisfying the first condition is $O(x/\ell^2)$, while the count of $n\le x$ satisfying the second is $O(x(\log_2{x})^2/\ell^2)$ (cf.\ the handling of condition (iii) in the proof of Lemma \ref{lem:inconvenientsolutions}). Summing on $\ell>w$, we see that our coprimality condition excludes only $o(x/w)$ values of $n$, which is acceptable as $x/w = o(N/\phi(q))$. 

Recall our notation $s_w(\cdot)$ for the $w$-smooth part. When our coprimality condition \eqref{eq:coprimecondition} holds,
\begin{multline} \lcm[\lambda(m), P_1-1,\dots, P_J-1]
\\ = \lcm[s_w(\lambda(m)), s_w(P_1-1),\dots, s_w(P_J-1)] \frac{\lambda(m)}{s_w(\lambda(m))} \prod_{j=1}^{J} \frac{P_j-1}{s_w(P_j-1)}.\label{eq:newlcm} 
\end{multline}
So at the cost of an error of $o(N/\phi(q))$, we may  swap out the lcm congruence condition in \eqref{eq:hugedoublesum} for the condition that the right-hand side of \eqref{eq:newlcm} be congruent to $a$ modulo $q$. We then remove the ordering on the $P_j$, which introduces a factor of $1/J!$, and we partition the inner sum on $P_1,\dots, P_J$ according to the values of $d_j:=s_w(P_j-1)$. The upshot is as follows:  Let $\mathbf{d} = (d_1,\dots,d_J)$ and define rational numbers $r_{m,\textbf{d}}$ and $a_{m,\textbf{d}}$ by 
\begin{equation}\label{eq:radefs} r_{m,\textbf{d}} = \lcm[s_w(\lambda(m)),d_1,\dots,d_J] \frac{\lambda(m)}{s_w(\lambda(m))} \prod_{j=1}^{J}\frac{1}{d_j}, \quad a_{m,\textbf{d}} = a/r_{m,\textbf{d}}. \end{equation} Then the count described by \eqref{eq:hugedoublesum} is equal --- up to an error of $o(N/\phi(q))$ --- to 
\begin{equation}\label{eq:hugerdoublesum} \frac{1}{J!}\sum_{\substack{m\le x \\ (\lambda(m),q)=1}} \sum_{\substack{d_J,\dots,d_1\text{ $w$-smooth} \\ (d_j,q)=1~\forall j \\ 2\mid d_j~\forall j }}    \sum_{\substack{P_J, \dots, P_1 \text{ distinct} \\ L_m < P_j~\forall j\\ P_J \cdots P_1  \le x/m \\ s_w(P_j-1)=d_j~\forall j \\ \prod_{j=1}^{J} (P_j-1) \equiv a_{m,\textbf{d}} \pmod{q}}} 1. \end{equation}
(The subscripted condition $2\mid d_j$ could be omitted at this stage but will prove useful later.) Note that $a_{m,\textbf{d}}$ determines a well-defined coprime residue class mod $q$, since $a_{m,\textbf{d}}$ is a rational number with numerator and denominator prime to $q$. 

Following \cite{PSR23}, we view the condition on $\prod_{j=1}^{J} (P_j-1)$ in \eqref{eq:hugerdoublesum} as cutting out a collection of possible tuples $(P_1,\dots,P_J)$ mod $q$. Let
\begin{equation}\label{eq:vmdef} \V_{m,\bf{d}}(q) = \{(a_1,\dots,a_J)\bmod{q}: \gcd\bigg(\prod_{i=1}^{J} a_j(a_j-1),q\bigg)=1, \prod_{j=1}^{J} (a_j-1) \equiv a_{m,\textbf{d}}\bmod{q}\}. \end{equation}
Then \eqref{eq:hugerdoublesum} can be rewritten as
\begin{equation}\label{eq:hugefundamental} \frac{1}{J!}\sum_{\substack{m \le x \\ (\lambda(m),q)=1}} \sum_{\substack{d_J,\dots,d_1\text{ $w$-smooth} \\ (d_j,q)=1~\forall j \\ 2\mid d_j~\forall j }} \sum_{\v \in \V_{m,\textbf{d}}(q)}   \sum_{\substack{P_J, \dots, P_1 \text{ distinct}\\ L_m < P_j~\forall j\\ P_J \cdots P_1  \le x/m \\ s_w(P_j-1)=d_j~\forall j \\ P_j \equiv a_j\bmod{q}~\forall j}} 1; \end{equation}
this expression \eqref{eq:hugefundamental} will serve as our basic approximation to the left-hand side of \eqref{eq:convenientreduced}.

Analogous manipulations reveal that the count of convenient $n\le x$ with $(\lambda(n),q)=1$ is (precisely) equal to
\begin{equation}\label{eq:hugefundamental-tocompare}  \frac{1}{J!}\sum_{\substack{m \le x \\ (\lambda(m),q)=1}} \sum_{\substack{d_J,\dots,d_1\text{ $w$-smooth} \\ (d_j,q)=1~\forall j \\ 2\mid d_j~\forall j }} \sum_{\substack{P_J, \dots, P_1 \text{ distinct}\\ L_m < P_j~\forall j\\ P_J \cdots P_1  \le x/m \\ s_w(P_j-1)=d_j~\forall j}} 1. \end{equation}
Rather than estimate \eqref{eq:hugefundamental} or \eqref{eq:hugefundamental-tocompare} directly, our strategy to prove \eqref{eq:convenientreduced} is to directly compare \eqref{eq:hugefundamental} and \eqref{eq:hugefundamental-tocompare}.

View \eqref{eq:hugefundamental} as having the sum on $P_1$ innermost. We now argue that (in a precise sense) on average over the $m,\mathbf{d}, \mathbf{v}$, and $P_J,\dots,P_2$ from \eqref{eq:hugefundamental},
\begin{equation}\label{eq:innersumapprox} \sum_{\substack{L_m < P_1 \le x/m P_2 \cdots P_{J} \\ P_1 \ne P_2,\dots, P_J \\ P_1 \equiv a_1\pmod{q} \\ s_w(P_1-1) = d_1}} 1 \approx \frac{1}{\phi_2(q)}\sum_{\substack{L_m < P_1 \le x/m P_2 \cdots P_{J} \\ P_1 \ne P_2,\dots, P_J \\ s_w(P_1-1) = d_1}} 1. \end{equation}  The left-hand side of \eqref{eq:innersumapprox} is the sum on $P_1$ in \eqref{eq:hugefundamental}; thus, we are asserting with this approximation that the mod $q$ congruence condition on $P_1$ in \eqref{eq:hugefundamental} can essentially be removed by insertion of a factor of $1/\phi_2(q)$. To argue this will assume that \begin{equation}\label{eq:mayassume} L_m < x/mP_2\cdots P_j; \end{equation} otherwise \eqref{eq:innersumapprox} is (trivially) an equality.

To save on notation, write $d$ for $d_1$ and $a$ for $a_1$. We start by estimating for each $T \ge y$ the count of $p\le T$ with $p\equiv a\pmod{q}$ and $s_w(p-1)=d$.  This is a sieve problem: We start with the set $\mathcal{P}$ of primes $p\le T$ with $p\equiv a\pmod{q}$ and $p\equiv 1\pmod{d}$, which we view as having approximate size $\mathcal{X}:= \Li(T)/\phi(d)\phi(q)$. To enforce the condition $s_w(p-1)=d$, we remove (sieve out) those $p \in \mathcal{P}$ for which $\frac{p-1}{d}$ is divisible by a prime at most $w$. By the definition of $\V_{m,\textbf{d}}(q)$, the integer $a-1$ is coprime to $q$, and so any prime divisor of $\frac{p-1}{d}$ must be coprime to $q$. Thus, we need only sieve by those primes up to $w$ not dividing $q$. 

Let $e$ be a squarefree, $w$-smooth number with $(e,q)=1$. The count of $p\in \mathcal{P}$ for which $e\mid \frac{p-1}{d}$ is precisely the count of primes up to $T$ belonging to a certain  coprime progression mod $deq$; it is therefore
\begin{equation}\label{eq:Aeest} \frac{\Li(T)}{\phi(de) \phi(q)} + O(E(T;q de)), \end{equation}
where (anticipating an  application of the Bombieri--Vinogradov theorem) we have set
\[ E(X;M) := \max_{2\le Y \le X} \max_{\substack{A \bmod{M} \\ (A,M)=1}} \left|\pi(Y;M,A)-\frac{\Li(Y)}{\phi(M)}\right|. \]
Setting $g(e):= \phi(d)/\phi(de)$, the main term in \eqref{eq:Aeest} is $\mathcal{X}g(e)$. Now $g$ is multiplicative in $e$, and $g$ satisfies the Iwaniec condition (with $\kappa=1$): If $\frac{3}{2} \le w_1 \le w_2 \le w$, then $\prod_{w_1 < p \le w_2,~p\nmid q}(1-g(p))^{-1} \le (1+O(1/\log{w_1})) \frac{\log{w_2}}{\log{w_1}}$, where the implied constant is absolute.  So we are set up to apply the Fundamental Lemma of the Sieve (see, e.g., \cite[Theorem 18.11, p.\ 190]{K19}), which yields
\begin{multline}\label{eq:sievebound} \#\{p\le T: p\equiv a \pmod{q},~s_w(p-1)=d\} \\ = \Bigg(\frac{\Li(T)}{\phi(d) \phi(q)} \prod_{\substack{p \le w \\ p\nmid q}}(1-g(p))\Bigg)  \left(1+O(u^{-u/2})\right) + O\left(\sum_{e\le w^{u}} E(T;qd e)\right), \end{multline}
where $u\ge 1$ is a parameter at our disposal. Letting
\[ z:= \exp((\log{x})^{1/4}), \] 
we choose $u$ so that $w^u=z$ (i.e., $u=(\log{x})^{1/8}$). Inserting this choice of $u$ and simplifying, the right-hand side of \eqref{eq:sievebound} becomes
\begin{multline}\label{eq:sieveoutput}\left(\frac{1}{2}\prod_{2 < p \le w}\frac{p-2}{p-1}\right) \frac{\Li(T)}{\phi(d)\phi_2(q)} \cdot \frac{\phi(d_{\odd})^2}{d_{\odd} \phi_2(d_\odd)}  \\ + O\bigg(\frac{T}{\phi(d)\phi(q)} \exp(-(\log{x})^{1/8}) + \sum_{e\le z} E(T;q d e)\bigg), \end{multline}
where $d_{\odd}$ denotes the largest odd divisor of $d$. 

If we forget the mod $q$ congruence condition on $p$ and  estimate the total number of $p\le T$ with $s_w(p-1)=d$, we obtain an identical estimate to \eqref{eq:sieveoutput}, but with all occurrences of $q$ replaced by $1$. Thus, writing $\Delta=\Delta(m,\mathbf{d},\mathbf{v}, P_J,\dots,P_2)$ for the magnitude of the difference between the two sides of \eqref{eq:innersumapprox}, 
\begin{equation}\label{eq:deltabound1} \Delta \ll \Delta_1 + \Delta_2 + \Delta_3,
\end{equation}
where
\begin{multline*} \Delta_1 =  J+ \frac{x \exp(-(\log{x})^{1/8})}{m P_2\cdots P_{J} \phi(d) \phi_2(q)}, \quad
\Delta_2 = \sum_{e \le z} E\bigg(\frac{x}{m P_2 \cdots P_{J}};qde\bigg),\quad\text{and} \\
\Delta_3 = \frac{1}{\phi_2(q)} \sum_{e \le z} E\bigg(\frac{x}{m P_2 \cdots P_{J}};de\bigg
).\end{multline*}
When $d$ is large, the bound
\begin{equation}\label{eq:deltabound2} \Delta \ll x/m d P_2\cdots P_J \end{equation}
will be more useful than \eqref{eq:deltabound1}; \eqref{eq:deltabound2} is obvious, since the count of primes $1$ mod $d$ to $x/mP_2\cdots P_J$ is smaller than $x/md P_2\cdots P_J$.

Let us show, using \eqref{eq:deltabound1} and \eqref{eq:deltabound2}, that
\begin{equation}\label{eq:deltasumbound} \sum_{m,\mathbf{d},\mathbf{v},P_J,\dots,P_2} \Delta(m,\mathbf{d},\mathbf{v},P_J,\dots,P_2) \ll x/(\log{x})^{A+2}, \end{equation}
where $m,\mathbf{d},\mathbf{v},P_J,\dots,P_2$ are restricted by the conditions of summation in \eqref{eq:hugefundamental} and by  \eqref{eq:mayassume}.

We take first the cases when $d \le z$. Here we appeal to \eqref{eq:deltabound1}, considering separately the contributions from $\Delta_1, \Delta_2, \Delta_3$. It is easy to handle $\Delta_1$: Since $x/mP_2\cdots P_J > L_m \ge y$, we have $J \ll \frac{x \exp(-(\log{x})^{1/8})}{m P_2\cdots P_{J} \phi(d) \phi_2(q)}$. Summing $\frac{x \exp(-(\log{x})^{1/8})}{m P_2\cdots P_{J} \phi(d) \phi_2(q)}$ on all $m\le x$, all primes $P_2,\dots,P_J \le x$ and all $d\le z$ gives a quantity of size $O(\frac{x}{\phi_2(q)}\exp(-(\log{x})^{1/9}))$. Since $P_2,\dots,P_J$ and $d$ together determine $\mathbf{d}$ and the components $a_2,\dots,a_{J}$ of $\mathbf{v}$, and since there are at most $\phi_2(q)$ choices for $a_1 \bmod{q}$, we conclude that $\Delta_1$ contributes only $O(x \exp(-(\log{x})^{1/9}))$ to the left-hand side of \eqref{eq:deltasumbound}. 

Since $\phi_2(q) \le q \le (\log{x})^{A}$, the $\Delta_2$ piece will be satisfactorily handled if
\begin{equation}\label{eq:willbe} \sum_{m \le x} \sum_{\substack{P_2,\dots,P_J \\ P_2 \cdots P_J \le x/my}} \sum_{d\le z} \sum_{e \le z} E\bigg(\frac{x}{m P_2 \cdots P_{J}};qde\bigg) = O(x/(\log{x})^{2A+2}). \end{equation}
(We use again that $P_2,\dots,P_J$ and $d$ determine $\mathbf{d}$ and $a_2,\dots,a_{J}$. Note that the restrictions in \eqref{eq:willbe} on $m, P_2, \dots, P_J$, and $d$ are implied by the conditions of summation in \eqref{eq:hugefundamental}, the assumption \eqref{eq:mayassume}, and our present stipulation that $d\le z$.) Trivially, $E(x/m P_2\cdots P_J; M) \ll x/\phi(M) m P_2\cdots P_J$ whenever $M\le x/m P_2\cdots P_J$. So by Cauchy--Schwarz and the Bombieri--Vinogradov theorem (bearing in mind that $qz^2 < y^{1/3} \le (x/mP_2\cdots P_J)^{1/3}$), 
\begin{align*} \sum_{d, e \le z} E\bigg(\frac{x}{m P_2 \cdots P_{J}};qde\bigg) &\le \sum_{M \le q z^{2}} \tau(M) \cdot E\bigg(\frac{x}{m P_2 \cdots P_{J}};M\bigg) \\ 
&\ll \Bigg(\frac{x}{m P_2\cdots P_J} \sum_{M \le qz^2} \frac{\tau(M)^2}{\phi(M)}\Bigg)^{1/2} \left(\sum_{M \le qz^2} E\bigg(\frac{x}{m P_2 \cdots P_{J}};M\bigg) \right)^{1/2} \\
&\ll \Bigg(\frac{x (\log{x})^4}{m P_2\cdots P_J}\Bigg)^{1/2} \Bigg(\frac{x}{m P_2\cdots P_J (\log{x})^{4A+12}}\Bigg)^{1/2}, \end{align*}
which is $O((x/mP_2\cdots P_J) (\log{x})^{-2A-4})$. Summing on $m\le x$ and $P_2, \dots, P_J\le x$, this is $O(x/(\log{x})^{2A+2})$, which is acceptable for \eqref{eq:willbe}. The errors induced by $\Delta_3$ can be treated entirely analogously; we leave the tedious but inglorious details to the reader. 

To handle the contributions to \eqref{eq:deltasumbound} from cases where $d> z$, it is enough (appealing now to \eqref{eq:deltabound2}) to show that
\begin{equation}\label{eq:largedpieces}\sum_{m \le x} \sum_{P_2,\dots,P_J \le x} \sum_{\substack{d> z \\ P^{+}(d)\le w}} \frac{x}{md P_2\cdots P_J}= O(x/(\log{x})^{2A+2}). \end{equation}
By \eqref{eq:cleanbound}, $\sum_{d > z,~P^{+}(d)\le w} 1/d \ll (\log{w}) \exp(-\frac{1}{2}(\log{x})^{1/8})$. This yields \eqref{eq:largedpieces} after crudely bounding the sums on $m, P_2,\dots, P_J$. Thus, we have \eqref{eq:deltasumbound}.

We conclude from \eqref{eq:deltasumbound} that
\begin{multline*} 
  \frac{1}{J!}\sum_{\substack{m \le x \\ (\lambda(m),q)=1}} \sum_{\substack{d_J,\dots,d_1\text{ $w$-smooth} \\ (d_j,q)=1~\forall j \\ 2\mid d_j~\forall j }} \sum_{\v \in \V_{m,\textbf{d}}(q)}   \sum_{\substack{P_J, \dots, P_1 \text{ distinct}\\ L_m < P_j~\forall j\\ P_J \cdots P_1  \le x/m \\ s_w(P_j-1)=d_j~\forall j \\ P_j \equiv a_j\bmod{q}~\forall j}} 1 \\ = \frac{1}{\phi_2(q)}\frac{1}{J!}\sum_{\substack{m \le x \\ (\lambda(m),q)=1}}  \sum_{\substack{d_J,\dots,d_1\text{ $w$-smooth} \\ (d_j,q)=1~\forall j \\ 2\mid d_j~\forall j }} \sum_{\v \in \V_{m,\textbf{d}}(q)}   \sum_{\substack{P_J, \dots, P_1 \text{ distinct}\\ L_m < P_j~\forall j\\ P_J \cdots P_1  \le x/m \\ s_w(P_j-1)=d_j~\forall j \\ P_j \equiv a_j\bmod{q}~\forall j\ge 2}} 1 + O(x/(\log{x})^{A+2}).
\end{multline*}
In fact, our arguments justify an error term of $O(x/J! (\log{x})^{A+2})$, but the extra factor of $J!$ is otiose here. 

Proceeding in exactly the same manner, we can successively remove the mod $q$ congruence conditions on $P_2, P_3, \dots, P_J$. At each step, we introduce a new factor of $1/\phi_2(q)$ and a new error of size $O(x/(\log{x})^{A+2})$. (Actually the error introduced after the $j$th step could be estimated as $O(x/J! \phi_2(q)^{j-1}(\log{x})^{A+2})$, but $O(x/(\log{x})^{A+2})$ is all we need.) Since $Jx/(\log{x})^{A+2} = o(N/\phi(q))$, we  make the following deduction: Up to an error of $o(N/\phi(q))$, the count of convenient $n\le x$ with $\lambda(n)\equiv a\pmod{q}$ is 
\begin{equation}\label{eq:conditionsremoved} \frac{1}{\phi_2(q)^J}\frac{1}{J!}\sum_{\substack{m \le x \\ (\lambda(m),q)=1}} \sum_{\substack{d_J,\dots,d_1\text{ $w$-smooth} \\ (d_j,q)=1~\forall j \\ 2\mid d_j~\forall j }} \sum_{\substack{P_J, \dots, P_1 \text{ distinct}\\ L_m < P_j~\forall j\\ P_J \cdots P_1  \le x/m \\ s_w(P_j-1)=d_j~\forall j}} \#\V_{m,\textbf{d}}(q). \end{equation}

Up to this point, it was only necessary to assume that $q$ is odd. We now tack on the hypothesis that $3\nmid q$. Lemma \ref{lem:congruencecount} then implies that $\#\V_{m,\textbf{d}}(q) = (1+o(1)) \frac{\phi_2(q)^{J}}{\phi(q)}$, uniformly in the values of $m$ and $\textbf{d}$ from \eqref{eq:conditionsremoved}. Hence, the size of \eqref{eq:conditionsremoved} is $(1+o(1))/\phi(q)$ times that of \eqref{eq:hugefundamental-tocompare}. Since \eqref{eq:hugefundamental-tocompare} counts convenient $n\le x$ with $(\lambda(n),q)=1$, the relation \eqref{eq:convenientreduced} follows. This completes the proof of Theorem \ref{thm:WUD2}.

\section{Interlude}\label{sec:interlude} Our proof of Theorem \ref{thm:WUD2} does not yield weak uniform distribution when $3\mid q$. However, one can show by these methods that the only obstruction to weak uniform distribution mod $q$ arises from an obstruction modulo $3$. The following proposition makes this precise.

\begin{prop}\label{prop:interlude} Fix $A>0$. As $x\to\infty$, 
  \begin{equation}\label{eq:interlude} \#\{n\le x: \lambda(n)\equiv a\pmod{q}\} = \frac{2+o(1)}{\phi(q)} 
  \#\{n \le x: \lambda(n)\equiv a\pmod{3} \} + o(N/\phi(q)), \end{equation}
uniformly in the choice of coprime residue class $a\bmod{q}$, where $q$ is odd, $3\mid q$, and $q\le (\log{x})^{A}$.
\end{prop}
\noindent In the remainder of this section we describe how to adapt the proof of Theorem \ref{thm:WUD2} to obtain Proposition \ref{prop:interlude}. Note that by Lemmas \ref{lem:inconvenient0} and  \ref{lem:inconvenientsolutions}, to prove Proposition \ref{prop:interlude} it suffices to establish the ``convenient version'' of \eqref{eq:interlude}, where $n$ is restricted to covenient values in both sets.

We start from \eqref{eq:conditionsremoved}, which (as we know already) counts convenient solutions $n\le x$ to $\lambda(n)\equiv a\pmod{q}$, up to an error of $o(N/\phi(q))$. The new twist is that when $3\mid q$, the sets $\V_{m,\textbf{d}}(q)$ are no longer uniformly of size $\sim \phi_2(q)^{J}/\phi(q)$. In fact, since each $a$ with $a(a-1)$ coprime to $q$ has $a-1\equiv 1\pmod{3}$, the set $\V_{m,\mathbf{d}}(q)$ is empty unless $a_{m,\textbf{d}} \equiv 1\pmod{3}$. 

Suppose we are in the case that $a_{m,\textbf{d}} \equiv 1\pmod{3}$; here we say that m and \textbf{d} are \textsf{compatible}. Then $\#\V_{m,\textbf{d}}(q) = \#\V_{m,\textbf{d}}(q_0) \cdot \#\V'$, where
\[ \V' = \{(a_1,\dots,a_J)\bmod{3^k}: \text{each $a_j\equiv 2\pmod{3}$},~\prod_{j=1}^{J}(a_j-1) \equiv a_{m,\textbf{d}} \pmod{3^k}\}. \]
By Lemma \ref{lem:congruencecount}, $\#\V_{m,\textbf{d}}(q_0) = (1+o(1)) \phi_2(q_0)^{J}/\phi(q_0)$. Also, $\#\V' = (3^{k-1})^{J-1}$: every choice of $a_1,\dots,a_{J-1} \equiv 2\pmod{3}$ determines a unique $a_J$. As $(3^{k-1})^{J-1}= 2\phi_2(3^{k})^{J}/\phi(3^{k})$, we deduce that $\#\V_{m,\textbf{d}}(q) = (2+o(1)) \phi(q)^{J}/\phi(q)$ for compatible $m$ and $\textbf{d}$.

We can recast the compatibility condition by  by referring back to \eqref{eq:newlcm} and the definition \eqref{eq:radefs} of $a_{m,\textbf{d}}$: For $m, \textbf{d}$, and $P_1,\dots, P_J$ as in \eqref{eq:conditionsremoved},
\[ m\text{ and }\textbf{d}\text{ are compatible} \Longleftrightarrow \lambda(m P_J \cdots P_1) \equiv a\pmod{3}. \]
Here we have used that each $P_j-1 \equiv 1\pmod{3}$.   

Combining the observations of the previous two paragraphs, we find that \eqref{eq:conditionsremoved} has size
\begin{equation*}\frac{(2+o(1))}{\phi(q)} \left[\frac{1}{J!}\sum_{\substack{m \le x \\ (\lambda(m),q)=1}} \sum_{\substack{d_J,\dots,d_1\text{ $y$-smooth} \\ (d_j,q)=1~\forall j \\ 2\mid d_j~\forall j }} \sum_{\substack{P_J, \dots, P_1 \text{ distinct}\\ L_m < P_j~\forall j\\ P_J \cdots P_1  \le x/m \\ s_w(P_j-1)=d_j~\forall j \\ \lambda(mP_J \cdots P_1) \equiv a\pmod{3}}} 1. \right].\end{equation*} 
The bracketed expression is precisely analogous to \eqref{eq:hugefundamental-tocompare};  it counts convenient $n\le x$ with $(\lambda(n),q)=1$ and $\lambda(n) \equiv a\pmod{3}$. We have therefore shown the `convenient version' of \eqref{eq:interlude} and completed the proof of Proposition \ref{prop:interlude}.

\section{Uniformity to $\log_3{x}$ when $3\mid q$: Proof of Theorem \ref{thm:WUD1}}\label{sec:thm1proof} In view of Proposition \ref{prop:interlude}, Theorem \ref{thm:WUD1} follows from our next result.

\begin{thm}\label{thm:mod3} Let $a \in \{\pm 1\}$. As $x\to\infty$, 
  \begin{equation}\label{eq:interlude2} \#\{n\le x: (\lambda(n),q)=1, \lambda(n)\equiv a\pod{3}\} \sim \frac{1}{2} \#\{n\le x: (\lambda(n),q)=1\}, \end{equation}
uniformly for integers $q\le \log_3{x}$ where $q$ is odd and $3\mid q$.
\end{thm}

Our strategy for proving Theorem \ref{thm:mod3} shares many features with the proof of Theorem \ref{thm:WUD1} but we must select certain parameters rather differently. In particular, the reader should be warned that $y, w$, and $L_m$ will be (re)used with different (but related) meanings.

We start by refreshing our notion of convenient. We will now call $n\le x$ \textsf{convenient} if $P^{+}(n) > y:=x^{1/\sqrt{\log_3{x}}}$ and $P^{+}(n)^2\nmid n$. Thus, $n$ is convenient precisely when we can write
\begin{equation}\label{eq:newconvenient} n = mP, \quad\text{where}\quad L_m:=\max\{P^{+}(m),y\} < P. \end{equation}

We continue to use $N$ to denote the count of $n\le x$ with $(\lambda(n),q)=1$. The following statement is the analogue of Lemma \ref{lem:inconvenient0} for our new definition of inconvenient.

\begin{lem}\label{lem:fewinconvenient} The number of inconvenient $n\le x$ with $(\lambda(n),q)=1$ is $o(N)$. 
\end{lem}

\begin{proof}The number of $n\le x$ divisible by the square of a prime exceeding $y$ is $o(x/y)$ and thus also $o(N)$. All other inconvenient $n\le x$ with $(\lambda(n),q)=1$ are $y$-smooth and divisible only by $p$ with $(p-1,q)=1$. The sieve bounds the number of such $n\le x$ as \begin{multline*} \ll x \prod_{\substack{p \le y \\ (p-1,q)>1}} \left(1-\frac{1}{p}\right) \prod_{y < p \le x}\left(1-\frac{1}{p}\right) \ll \frac{x}{\log{x}} \prod_{\substack{p \le y \\ (p-1,q)=1}} \left(1-\frac{1}{p}\right)^{-1} \\
  \ll \frac{x}{\log{x}} \exp\Bigg(\sum_{\substack{p \le y \\ (p-1,q)=1}} \frac{1}{p}\Bigg) = \frac{x}{(\log{x})^{1-\alpha}} \exp(O((\log_2(3q))^{2})) \cdot \frac{1}{(\log_3{x})^{\alpha/2}},\end{multline*}
where we used Proposition \ref{prop:slightgen} to estimate on the sum on $p$. Comparing with \eqref{eq:sumcoprime}, keeping in mind the lower bound \eqref{eq:alphalower} and the upper bound $q\le \log_3{x}$, this is seen to be $o(N)$. (In fact, it is $o(N/\exp(\sqrt{\log_4 x}))$.)
\end{proof}

In view of Lemma \ref{lem:fewinconvenient}, Theorem \ref{thm:mod3} will be established once it is shown that
\begin{multline}\label{eq:mod3reduced} \#\{\text{convenient } n\le x: (\lambda(n),q)=1, \lambda(n)\equiv a\pmod{3}\}\\ = \frac{1}{2} \#\{\text{convenient } n\le x: (\lambda(n),q)=1\} + o(N). \end{multline}
For convenient $n$ with $n=mP$ as in \eqref{eq:newconvenient},
\begin{equation}\label{eq:lambdamod3eq}\lambda(n) = \frac{\lambda(m) (P-1)}{(\lambda(m),P-1)}. \end{equation}

The following ``near-identity'' for $(\lambda(m),P-1)$ is the linchpin of our proof. Let
\[ w := \log_2{x}. \]

\begin{lem}\label{lem:gcdformula} Among all convenient $n\le x$ with $(\lambda(n),q)=1$, all but $o(N)$ satisfy \[ (\lambda(m),{P-1}) = s_w(P-1). \]
\end{lem}

\begin{proof} First, we may restrict attention to convenient $n = mP\le x$ for which
\begin{enumerate}
  \item[(i)] $m$ is not divisible by the square of a prime exceeding $\log_3{x}$.
\end{enumerate}
Since $y < P \le x/m$, the number of convenient $n$ corresponding to a given $m$ is bounded by $\pi(x/m)$, which is $\ll x/m\log{(x/m)} \ll x/m\log{y} = x\sqrt{\log_3{x}}/m \log{x}$. Thus, an upper bound for the number of exceptions to (i) follows from an upper bound for $\sum 1/m$. 

Assuming (i) fails, write $m = \ell^2 m_0$ where $\ell > \log_3{x}$. As $(\lambda(m),q)=1$, every $p$ dividing $m_0$ has $(p-1,q)=1$. Hence, the reciprocal sum of possible $m_0$ satisfies $\sum 1/m_0 \ll \exp(\sum_{p \le x,~(p-1,q)=1} 1/p) \ll (\log{x})^{\alpha}\exp(O((\log_2(3q))^2))$. Also, $\sum_{\ell > \log_3 x}1/\ell^2 \ll 1/\log_3{x}$. It follows that the number of exceptions to (i) is .
\[ \ll \frac{x}{(\log{x})^{1-\alpha}}\exp(O((\log_2(3q))^2)) \cdot \frac{\sqrt{\log_3{x}}}{\log_3 x}. \] This is $o(N)$ (in fact, $o(N/(\log_3{x})^{1/3})$) and so is acceptable for us.

We may further assume that
\begin{enumerate}
    \item[(ii)] $\lambda(m)$ is divisible by every positive integer $s\le w':= \log_2{x}/(\log_3{x})^2$ with $(s,q)=1$.
\end{enumerate}
Let $s$ be a positive integer with $s\le w'$, $(s,q)=1$, and suppose that $s\nmid \lambda(m)$. Then $m$ is divisible only by primes $p$ with $p-1$ coprime to $q$ and $p-1$ not divisible by $s$. Invoking Proposition \ref{prop:slightgen},
\[ \sum \frac{1}{m} \ll \exp\Bigg(\sum_{\substack{p\le x \\ (p-1,q)=1 \\ s\nmid p-1}}\frac{1}{p}\Bigg)\ll (\log{x})^{\alpha(1-1/\phi(s))} \exp(O((\log_2{(3q)})^2)).\]
Reasoning as in (i), the number of corresponding $n$ is therefore
\begin{equation}\label{eq:iexcept} \ll \frac{x}{(\log{x})^{1-\alpha}} \exp(O((\log_2{(3q)})^2)) \cdot \frac{\sqrt{\log_3{x}}}{(\log{x})^{\alpha/\phi(s)}}.\end{equation}
As $s\le w'$, 
\[ (\log{x})^{\alpha/\phi(s)} \ge (\log{x})^{\alpha/s}\ge \exp(\alpha (\log_3{x})^2),\] and so the right-hand side of \eqref{eq:iexcept} is $o(N/\exp((\log_{3}{x})^{3/2}))$ (say). Summing on the at most $w'$ values of $s$, we obtain only $o(N)$ exceptions to (ii).

Continuing, we can assume that
\begin{enumerate}
  \item[(iii)] $(\lambda(m),P-1)$ has all prime factors at most $w'':= \log_2{x} (\log_3{x})^2$.
\end{enumerate}
If (iii) fails, choose a prime prime $\ell > w''$  dividing $\lambda(m)$ and $P-1$. Our strategy will be to count possibilities for $P$ given $m, \ell$ and then sum on $m,\ell$. If we assume temporarily that $\ell \le y^{1/3}$, we can apply Brun--Titchmarsh: The number of $P\equiv 1\pmod{\ell}$ with $y < P \le x/m$ is
\[ \ll \frac{x}{m\ell \log(x/m\ell)} \ll \frac{x \sqrt{\log_3{x}}}{m \ell \log{x}}. \]
Since $\ell \mid \lambda(m)$ and (by (i)) $\ell^2\nmid m$, we must be able to write $m = p m_0$, where $p\equiv 1\pmod{\ell}$. Since $\sum 1/p \ll \log_2{x}/\ell$ (taken over all $p\le x$ with $p\equiv 1\pmod{\ell}$) and $\sum 1/m_0 \ll (\log{x})^{\alpha} \exp(O((\log_2{(3q)})^2))$ (taken over all $m_0\le x$ divisible only by $p$ with $(p-1,q)=1$), we conclude that the count of exceptional $n$ arising from a given $\ell \le y^{1/3}$ is 
\[ \ll \frac{x}{(\log{x})^{1-\alpha}} \exp(O((\log_2{(3q)})^2)) \cdot \frac{\log_2{x} \sqrt{\log_3{x}}}{\ell^2}. \] 
Summing on $\ell > w''$ gives a quantity that is $o(N)$, in fact $o(N/\log_3{x})$. Above, we ignored the cases when $\ell > y^{1/3}$. These are handled similarly, except that in place of Brun--Titchmarsh one bounds the number of $P$ trivially by $x/m\ell$. We leave to the reader to check that these cases contribute $o(N)$ (in fact, $o(N/y^{1/4})$). 

We can also assume that
\begin{enumerate}
\item[(iv)] $P-1$ has no prime factors in the interval $(w',w'']$.
\end{enumerate}
If (iv) fails, choose $\ell \mid P-1$ with $\ell \in (w',w'']$. Arguing with Brun--Titchmarsh as in (iii), the number of $P$ given $\ell,m$ is $\ll x\sqrt{\log_3{x}}/m\ell \log{x}$. Summing on $m$, the number of exceptions to (iv) arising from a given $\ell$ is 
\[ \ll \frac{x}{(\log{x})^{1-\alpha}} \exp(O((\log_2{(3q)})^2)) \cdot \frac{\sqrt{\log_3{x}}}{\ell}. \]
Since $\sum_{\ell \in (w,w'']} 1/\ell \ll \log_4{x}/\log_3{x}$, there are $o(N)$ exceptions to (iv) (in fact, $o(N/(\log_3{x})^{1/3}))$). 

Conditions (iii) and (iv) ensure that $(\lambda(m),P-1)$ divides $s_w(P-1)$. Furthermore, if $(\lambda(m),P-1) \ne s_w(P-1)$, then there is a prime power $p^k$, with $p \le w'$, $p\nmid q$, such that $p^k \mid P-1$ but $p^k\nmid \lambda(m)$. By (ii), $p^k > w'$, and so $k > 1$. Write $s=p^k$ and suppose to start with that $s \le y^{1/3}$. Arguing as in (iv) the number of $n$ corresponding to a given $s$ is 
\[ \ll \frac{x}{(\log{x})^{1-\alpha}} \exp(O((\log_2{(3q)})^2)) \cdot \frac{\sqrt{\log_3{x}}}{s}. \]
This is $o(N)$ after being summed on proper prime powers $s > w'$; in fact, the sum on all \emph{squarefull} $s > w'$ is $o(N/(\log_2{x})^{1/3})$. When $s> y^{1/3}$, we bound the count of $P$ given $m,s$ by $x/ms$. Summing on all $m\le x$ and all squarefull $s > y^{1/3}$ gives a contribution of size $o(x/y^{1/7})$ (say), which is certainly also $o(N)$. This completes the proof of Lemma \ref{lem:gcdformula}.
\end{proof}

If $n\le x$ is convenient with $(\lambda(n),q)=1$, and $n$ satisfies the conclusion of Lemma \ref{lem:gcdformula}, then (by \eqref{eq:lambdamod3eq}) $\lambda(n) \equiv \lambda(m)(P-1)/s_w(P-1) \equiv \lambda(m)/s_w(P-1)\pmod{3}$. It follows that the count of convenient $n\le x$ with $(\lambda(n),q)=1$ and $\lambda(n)\equiv a\pmod{3}$ is given by 
\begin{equation}\label{eq:mod3-0} \sum_{\substack{m\le x \\ (\lambda(m),q)=1}} \sum_{\substack{d\text{ $w$-smooth} \\ d\equiv a\lambda(m)\pmod{3} \\ (d,q)=1,~2\mid d}} \sum_{\substack{L_m < P \le x/m \\ s_w(P-1)=d}} 1, \end{equation}
up to an error of $o(N)$. 

This sum on $P$ in \eqref{eq:mod3-0} is reminscent of \eqref{eq:innersumapprox} but (since $w$ is now minuscule) can be estimated satifactorily by a direct application of inclusion-exclusion. Let 
\begin{equation}\label{eq:Wgdef} W = \frac{1}{2}\prod_{2 < p \le w} \frac{p-2}{p-1}, \quad\text{and define, for odd $d$,}\quad g(d) = \frac{\phi(d)^2}{d \phi_2(d)}. \end{equation}
A straightforward calculation shows that, as long as $L_m < x/m$, 
\begin{align*} \sum_{\substack{L_m < P \le x/m \\ s_w(P-1)=d}} 1 &= \sum_{e\, \mid\, \prod_{p \le w} p}\mu(e) \left(\pi(x/m;1,de)-\pi(L_m;1,de)\right) \\
&=(\Li(x/m)-\Li(L_m))W \frac{g(d_\odd)}{\phi(d)} + O\left(\sum_{e\, \mid\, \prod_{p \le w} p} \tilde{E}(x/m;de)\right), \end{align*}
where
\[ \tilde{E}(X;M) = \max_{2 \le Y \le X} \left|\pi(Y;M,1) - \frac{\Li(Y)}{\phi(M)}\right|.  \]
This estimate might be compared with with \eqref{eq:sieveoutput}; note that $\phi_2(q)$ is absent from the denominator of our new main term as there is no mod $q$ congruence condition on $P$.

The contribution of the error terms when summed on the $m, d$ from \eqref{eq:mod3-0} can be estimated as in our discussion of \eqref{eq:innersumapprox}: We use Bombieri--Vinogradov and Cauchy--Schwarz to treat cases when $d\le y^{1/3}$; these total $O(x/(\log{x})^{B})$ for any fixed $B$. To handle $d > y^{1/3}$, we use the trivial bound $\tilde{E}(x/m;de) \ll x/m\phi(de)$, valid for all $d,e$. (It is important here that we work with $\tilde{E}(x/m;de)$ rather than $E(x/m;de)$.) Since $x/m\phi(de) \le x/m\phi(d)\phi(e)$ and $d/\phi(d)\ll \log{w}$ for $w$-smooth $d$, 
\[ \sum_{m \le x} \sum_{\substack{d > y^{1/3}\\ P^{+}(d)\le w}} \sum_{e\,\mid\,\prod_{p\le w} p}\frac{x}{m\phi(de)} \ll x \log{w} \Bigg(\sum_{m \le x}\frac{1}{m}\Bigg)\Bigg(\sum_{e \text{ $w$-smooth}}\frac{1}{\phi(e)}\Bigg) \sum_{\substack{d > y^{1/3} \\ P^{+}(d)\le w}}\frac{1}{d}. \]
By \eqref{eq:cleanbound}, the sum on $d$ is $O(\exp(-(\log{x})^{9/10}))$. Since the sums on $m$ and $e$ are $O(\log{x})$ and $O(\log{w})$ respectively, we see that those $d> y^{1/3}$ also contribute $O(x/(\log{x})^{B})$. In particular (taking $B=1$, say), the error terms are in total $o(N)$. 

Therefore, up to an error of $o(N)$, the count of convenient $n\le x$ with $(\lambda(n),q)=1$ and $\lambda(n)\equiv a\pmod{3}$ is 
\begin{equation}\label{eq:mod3-1} W\sum_{\substack{m\le x \\ (\lambda(m),q)=1 \\ L_m < x/m}} (\Li(x/m)-\Li(L_m))\sum_{\substack{d\text{ $w$-smooth} \\ d\equiv a\lambda(m)\pmod{3} \\ (d,q)=1,~2\mid d}} \frac{g(d_{\odd})}{\phi(d)}. \end{equation}
Carrying out the same arguments, but forgetting the mod $3$ constraint on $\lambda(n)$, will show that the count of convenient $n\le x$ with $(\lambda(n),q)=1$ is
\begin{equation}\label{eq:mod3-2} W\sum_{\substack{m\le x \\ (\lambda(m),q)=1 \\ L_m < x/m}} (\Li(x/m)-\Li(L_m))\sum_{\substack{d\text{ $w$-smooth} \\ (d,q)=1,~2\mid d}} \frac{g(d_{\odd})}{\phi(d)}, \end{equation}
up to an error of $o(N)$. So to prove \eqref{eq:mod3reduced}, it is enough to show that \eqref{eq:mod3-1} is half of \eqref{eq:mod3-2}, to within $o(N)$. For this we zero in on the sums on $d$ in \eqref{eq:mod3-1} and \eqref{eq:mod3-2}.

It is expedient to group $d$ sharing the same odd part $d'$. For each $\epsilon \in \{\pm 1\}$,  
\begin{align*} \sum_{\substack{d\text{ $w$-smooth} \\ d\equiv \epsilon \pmod{3}\\ (d,q)=1,~2\mid d}} \frac{g(d_{\odd})}{\phi(d)} &= \sum_{\substack{d'\text{ $w$-smooth} \\ (d',2q)=1}} \frac{g(d')}{\phi(d')} \sum_{\substack{k \ge 1 \\ 2^k d'\equiv \epsilon \pmod{3}}}\frac{1}{\phi(2^k)}
\\
&= \frac{2}{3}\sum_{\substack{d'\text{ $w$-smooth} \\ d'\equiv \epsilon\pmod{3}\\ (d',2q)=1}} \frac{g(d')}{\phi(d')} + \frac{4}{3}\sum_{\substack{d'\text{ $w$-smooth} \\ d'\equiv -\epsilon\pmod{3} \\ (d',2q)=1}} \frac{g(d')}{\phi(d')} .\end{align*}
Hence, with $\chi$ denoting the nontrivial character mod $3$,
\[ \Bigg|\sum_{\substack{d\text{ $w$-smooth} \\ d\equiv 1 \pmod{3}\\ (d,q)=1,~2\mid d}} \frac{g(d_{\odd})}{\phi(d)} - \sum_{\substack{d\text{ $w$-smooth} \\ d\equiv -1 \pmod{3}\\ (d,q)=1,~2\mid d}} \frac{g(d_{\odd})}{\phi(d)}\Bigg|= \frac{2}{3}\Bigg|\sum_{\substack{d'\text{ $w$-smooth} \\ (d',2q)=1}}\frac{\chi(d') g(d')}{\phi(d')}\Bigg|.\]
We develop the sum on $d'$ as an Euler product, noting that $g(p)=g(p^2)=g(p^3)=\dots$:
\[ \sum_{\substack{d'\text{ $w$-smooth} \\ (d',2q)=1}}\frac{\chi(d') g(d')}{\phi(d')} = \prod_{\substack{p \le w \\ p \nmid 2q}} \left(1+\frac{\chi(p) g(p)}{p-1}\left(1-\frac{\chi(p)}{p}\right)^{-1}\right).\]
The $p$th term in our right-hand product has size $1+\chi(p)/p + O(1/p^2)$. Hence, the product is $\ll \exp(\sum_{p \le w,~p\nmid 2q} \chi(p)/p)$, which is $\ll \log_2{(3q)} \exp(\sum_{p \le w} \chi(p)/p) \ll \log_2{(3q)}$, using in the last step the convergence of $\sum_{p} \chi(p)/p$. It follows that the difference between \eqref{eq:mod3-1} and half of \eqref{eq:mod3-2} is 
\begin{multline*} \ll \log_{2}(3q) \cdot W \sum_{\substack{m\le x \\ (\lambda(m),q)=1 \\ L_m < x/m}} (\Li(x/m)-\Li(L_m)) \ll \log_{2}(3q) \cdot W \sum_{\substack{m\le x \\ (\lambda(m),q)=1}} \frac{x\sqrt{\log_3{x}}}{m\log{x}}\\ \ll \frac{x}{(\log{x})^{1-\alpha}} \exp(O((\log_2{(3q)})^2)) \cdot  W\log_2{(3q)} \sqrt{\log_3{x}}.\end{multline*}
As $W \ll 1/\log{w} \ll 1/\log_3{x}$ and $\log_2{(3q)}\ll \log_5{x}$, this last expression is $o(N/(\log_3{x})^{1/3})$ and so in particular $o(N)$. This completes the proof of Theorem \ref{thm:WUD1}.

\section{Leading digits of $\lambda(n)$ and $\ell_a(n)$}\label{sec:benford} The proofs of Theorems \ref{thm:Benford} and \ref{thm:Benford2} proceed along similar lines to our argument for Proposition \ref{prop:interlude}; in particular, we keep the definition \eqref{eq:newconvenient} of `convenient' and continue to write $w$ for $\log_2{x}$.

\subsection{Proof of Theorem \ref{thm:Benford}} By Lemma \ref{lem:fewinconvenient} with $q=1$, all but $o(x)$ values of $n\le x$ are convenient. Furthermore, among convenient $n= mP\le x$, all but $o(x)$ have $(\lambda(m),P-1) = s_w(P-1)$ (apply Lemma \ref{lem:gcdformula} with $q=1$). It follows that $\sum_{n\le x}\lambda(n)^{i\theta}$ is equal to
\begin{equation}\label{eq:lambdaexpr0} \sum_{\substack{m \le x\\ L_m < x/m}}  \lambda(m)^{i\theta} \sum_{\substack{d\text{ $w$-smooth} \\ 2\mid d}} \frac{1}{d^{i\theta}} \sum_{\substack{L_m < P \le x/m \\ s_w(P-1)=d}} (P-1)^{i\theta}, \end{equation}
up to an error of $o(x)$. 

As in \S\ref{sec:thm1proof}, the sum on $P$ can be estimated using inclusion-exclusion:
\[\sum_{\substack{L_m < P \le x/m \\ s_w(P-1)=d}} (P-1)^{i\theta} = \sum_{e\text{ $w$-smooth}} \mu(e) \sum_{\substack{P\equiv 1\pmod{de} \\ L_m < P \le x/m}} (P-1)^{i\theta}.\]
Summing by parts,
\[ \sum_{\substack{P\equiv 1\pmod{de} \\ L_m < P \le x/m}} (P-1)^{i\theta} = I_m/\phi(de) + O(\tilde{E}(x/m;de)\log{x}), \] 
where 
\[ I_m:= \int_{L_m}^{x/m} \frac{(t-1)^{i\theta}}{\log{t}}\, \mathrm{d}t. \]
(Here and below, we allow implied constants to depend on the fixed parameter $\theta$.) Plugging these estimates back  into \eqref{eq:lambdaexpr0}, the arguments of \S\ref{sec:thm1proof} will show that the accumulation of errors terms is, in total, $O(x/(\log{x})^{B})$, for any fixed $B$. 

After simplifying the main terms \eqref{eq:lambdaexpr0} becomes, up to an error of $o(x)$,
\[ W \sum_{\substack{m \le x \\ L_m < x/m}} \lambda(m)^{i\theta} I_m \sum_{\substack{d\text{ $w$-smooth} \\ 2\mid d}} \frac{g(d_{\odd})}{\phi(d) d^{i\theta}}, \]
with $W$ and $g(d)$ as in \eqref{eq:Wgdef}. Grouping $d$ sharing the same odd part $d'$ transforms this last expression into
\[ K_{\theta} W \sum_{\substack{m \le x \\ L_m < x/m}} \lambda(m)^{i\theta} I_m \sum_{\substack{d'\text{ $w$-smooth} \\ 2\nmid d'}} \frac{g(d')}{\phi(d') d'^{i\theta}}, \quad\text{where}\quad  K_{\theta} := \sum_{k \ge 1} \frac{1}{\phi(2^k)\cdot 2^{ki\theta}}. \] The sum on $d'$ can be developed into an Euler product with $p$th term $1+ p^{-1-i\theta} + O(p^{-2})$, for each odd $p\le w$; hence, that sum is $O(\exp(\Re\sum_{2<p \le w} p^{-1-i\theta}))$. By the prime number theorem with de la Vall\'ee Poussin error and Abel summation, the series $\sum_{p} p^{-1-i\theta}$ converges. (A well-timed use of integration by parts is helpful in this calculation.) In particular, the partial sums of $p^{-1-i\theta}$ are $O(1)$ (cf.\ \cite[Lemma 4.3]{CLPSR23}). Thus, our sum on $d'$ is also $O(1)$. 

Collecting estimates shows that $\sum_{n\le x} \lambda(n)^{i\theta}$ is, up to an error of $o(x)$,
\[ \ll W \sum_{\substack{m \le x \\ L_m < x/m}} |I_m| \ll W \sum_{\substack{m \le x \\ L_m < x/m}} \frac{x}{m \log(x/m)} \ll \frac{xW \sqrt{\log_3{x}}}{\log{x}}\sum_{m \le x}\frac{1}{m} \ll xW\sqrt{\log_3{x}} \ll \frac{x}{\sqrt{\log_3{x}}}. \]
This completes the proof of Theorem \ref{thm:Benford}.

\subsection{Proof of Theorem \ref{thm:Benford2}} We need two additional lemmas whose proofs are deferred momentarily. Throughout this section, all implied constants may depend on $a$. 

\begin{lem}\label{lem:ordlem1} Fix an integer $a$ with $|a| > 1$. Among all convenient $n=mP\le x$ with $(a,n)=1$, all but $o(x)$ satisfy 
  \[ (\ell_a(m), \ell_a(P)) = s_w(\ell_a(P)). \]
\end{lem}

\begin{lem}[conditional on GRH]\label{lem:ordlem2} Fix an integer $a$ with $|a| > 1$. Among all convenient $n=mP\le x$ with $(a,n)=1$, all but $o(x)$ are such that $(P-1)/\ell_a(P)$ is $w$-smooth.
\end{lem}

Fix $a \in \mathbb{Z}$ with $|a| > 1$. If $n\le x$ is convenient with $(a,n)=1$ and $n$ satisfies the conclusions of both Lemmas \ref{lem:ordlem1} and \ref{lem:ordlem2}, then 
\[ \ell_a(mP) = \lcm[\ell_a(m),\ell_a(P)] = \ell_a(m)\cdot \frac{\ell_a(P)}{s_w(\ell_a(P))} = \ell_a(m)\cdot \frac{P-1}{s_w(P-1)}. \]
From Lemmas \ref{lem:fewinconvenient} (with $q=1$), \ref{lem:ordlem1}, and \ref{lem:ordlem2},  we deduce that $\sum_{n\le x,~(n,a)=1} \ell_a(n)^{i\theta}$ coincides, up to an error of $o(x)$, with
\begin{equation}\label{eq:ordexpr0} \sum_{\substack{m \le x\\ (m,a)=1 \\ L_m < x/m}} \ell_a(m)^{i\theta} \sum_{\substack{d\text{ $w$-smooth} \\ 2\mid d}} \frac{1}{d^{i\theta}} \sum_{\substack{L_m < P \le x/m \\ s_w(P-1)=d}} (P-1)^{i\theta}. \end{equation}
This is nearly the same expression as \eqref{eq:lambdaexpr0}; the only differences are that $\lambda(m)$ has been replaced by $\ell_a(m)$ and that $m$ is now restricted to values coprime with $a$. It is straightforward to check that the method used to bound \eqref{eq:lambdaexpr0} still applies, so that \eqref{eq:ordexpr0} is $o(x)$ (in fact, $O(x/\sqrt{\log_3{x}})$). Thus, it remains only to prove Lemmas \ref{lem:ordlem1} and \ref{lem:ordlem2}.

For the proof of Lemma \ref{lem:ordlem1}, we recall a result of Pappalardi on the frequency of primes $p$ for which $\ell_a(p)$ is divisible by a given $d$.

\begin{prop}\label{prop:pap} Fix an integer $a$ with $|a| > 1$. For all $x\ge 3$, and all positive integers $d$,
  \begin{equation}\label{eq:papestimate} \#\{p\le x: p\nmid a, d \mid \ell_a(p)\} = \beta_a(d) \frac{x}{\log{x}}  + O\left(\tau(d)d\cdot x \left(\frac{(\log_2{x})^2}{\log{x}}\right)^{7/6}\right), \end{equation}
where $\beta_a(d)$ is a constant depending only on $a$ and $d$, and 
\begin{equation}\label{eq:moree} \beta_a(d) \asymp 1/d.\end{equation}
\end{prop}

The estimate \eqref{eq:papestimate} is a special case of \cite[Theorem 1]{pappalardi15} while \eqref{eq:moree} is implicit in the exact expression for $\beta_a(d)$ obtained by Moree in \cite[Theorem 2]{moree05}. We note that while the statement of Theorem 1 in \cite{pappalardi15} requires that $x\to\infty$, the estimate \eqref{eq:papestimate} holds trivially for $x\ge 3$ that are bounded in terms of $a$. 

Proposition \ref{prop:pap} implies the following analogue of Lemma \ref{lem:PN}, which seems of independent interest.

\begin{lem}\label{lem:orderrecip} Fix an integer $a$ with $|a| > 1$.  For each $x\ge 3$ and each positive integer $d$,
  \begin{equation}\label{eq:orderrecip} \sum_{\substack{p \le x,~p\nmid a \\ d \mid \ell_a(p)}} \frac{1}{p} = \beta_a(d)\log_2{x} + O\left(\frac{\log{(3d)}}{\phi(d)}\right), \end{equation}
  where $\beta_a(d)$ is as in the statement of Proposition \ref{prop:pap}.
\end{lem}

Lemma \ref{lem:orderrecip} is somewhat sharper than will be needed; it will suffice in our application that the $O$-term in \eqref{eq:orderrecip} is $O(1)$.

\begin{proof} First, observe that with $x_0 := \exp(2d^{30})$, 
\[ \sum_{\substack{p \le x_0,~p\nmid a \\ d \mid \ell_a(p)}} \frac{1}{p} \le \sum_{\substack{p \le x_0 \\ p\equiv 1\pmod{d}}} \frac{1}{p} \ll \frac{\log_2{x_0}}{\phi(d)} \ll \frac{\log(3d)}{\phi(d)}. \]
This implies \eqref{eq:orderrecip} when $3 \le x \le x_0$ (keeping in mind that $\beta_a(d) \ll 1/d$). So we suppose now that $x > x_0$. Let $N(t)$ be the number of primes $p$ up to $t$, not dividing $a$, for which $d\mid \ell_a(p)$. By Proposition \ref{prop:pap},
\begin{multline*} \sum_{\substack{x_0 < p \le x,~p\nmid a \\ d \mid \ell_a(p)}} \frac{1}{p} = \int_{x_0}^{x} \frac{1}{t} \,\mathrm{d}N(t) = \beta_a(d) \int_{x_0}^{x} \frac{\mathrm{d}t}{t\log{t}} \\ + O\left(\frac{\beta_a(d)}{\log{x_0}} + \tau(d) d \left(\frac{(\log_2{x_0})^2}{\log{x_0}}\right)^{7/6} + \tau(d) d \int_{x_0}^{x} \frac{1}{t}\left(\frac{(\log_2{t})^2}{\log{t}}\right)^{7/6}\,\mathrm{d}t\right). \end{multline*} The main term here has size $\beta_a(d) (\log_2{x}-\log_2{x_0}) = \beta_a(d) \log_2{x} + O(\log(3d)/d)$. Since $\tau(d)d \le d^2$ and $((\log_2{t})^2/\log{t})^{7/6} \ll (\log{t})^{-11/10}$, the $O$-terms are
\[ \ll \frac{1}{d\log{x_0}} + d^2 (\log{x_0})^{-11/10}+ d^2 \int_{x_0}^{x} t^{-1} (\log{t})^{-11/10}\,\mathrm{d}t \ll d^{-1}. \] Collecting estimates gives the lemma.
\end{proof}

\begin{proof}[Proof of Lemma \ref{lem:ordlem1}] We claim all but $o(x)$ convenient $n\le x$ having $(n,a)=1$ satisfy 
\begin{enumerate}
  \item[(ii$'$)] $\ell_a(m)$ is divisible by every positive integer $s\le w':=\log_2{x}/(\log_3{x})^2$, 
  \item[(iii$'$)]  $(\ell_a(m),\ell_a(P))$ has all prime factors at most $w'':=\log_2{x}(\log_3{x})^2$, \emph{and}
  \item[(iv$'$)] $\ell_a(P)$ has no prime factors in the interval $(w',w'']$.
\end{enumerate}
Conditions (iii$'$) and (iv$'$) are implied by conditions (iii) and (iv) from the proof of Lemma \ref{lem:gcdformula}, and so (from that proof, with $q=1$) these two conditions admit only $o(x)$ exceptions. Turning to (ii$'$), fix a positive integer $s\le w'$. If $s\nmid \ell_a(m)$, then there is no prime $p$ dividing $m$ for which $s\mid \ell_a(p)$. Hence,
\[ \sum_{m} \frac{1}{m} \le \prod_{\substack{p\le x,~p\nmid a \\ s\nmid \ell_a(p)}} \left(1-\frac{1}{p}\right)^{-1} \ll \log{x} \prod_{\substack{p\le x,~p\nmid a \\ s\mid \ell_a(p)}} \left(1-\frac{1}{p}\right) \ll (\log{x}) \exp(-\beta_a(s) \log_2{x}). \]
Given $m$, the number of possible $P$ is $O\left(\frac{x\sqrt{\log_3{x}}}{m\log{x}}\right)$. It follows that the number of convenient $n\le x$ with $(n,a)=1$ and $s\nmid \ell_a(n)$ is 
\[ \ll (x\sqrt{\log_3{x}}) \exp(-\beta_a(s) \log_2{x}) \ll x \exp(-(\log_3{x})^{3/2}), \]
using in the last step that $\beta_a(s) \gg 1/s \gg 1/w'$. Summing on $s\le w'$ shows that the number of exceptions to (ii$'$) is $o(x)$. 

We thus restrict to $n$ satisfying all of (ii$'$)--(iv$'$). By (iii$'$) and (iv$'$), $(\ell_a(m),\ell_a(P))$ is a divisor of $s_w(\ell_a(P))$. If it is a proper divisor, then there is a prime power $p^k$ with $p\le w'$ such that $p^k \mid \ell_a(P)$ but $p^k \nmid \ell_a(m)$. By (ii$'$), it must be that $k > 1$. Since $s\mid \ell_a(P) \mid P-1$, we deduce that $P-1$ is divisible by a proper prime power $s>w'$. The number of convenient $n$ of this kind is shown to be $o(x)$ in the paragraph concluding the proof of Lemma \ref{lem:gcdformula} (read now with $q=1$).
\end{proof}

Lemma \ref{lem:ordlem2} is very close in character to Li and Pomerance's Proposition 1 in \cite{LP03} and admits a nearly identical proof. The next three results are taken from \cite{LP03}.

\begin{lem}\label{lem:LPlem0} Fix an integer $a$ with $|a| > 1$. For all $x\ge 100$, the number of integers $n\le x$ divisible by a prime $p > w$ with $\ell_a(p) < p^{1/2}/\log{p}$ is $\ll x/\log_3{x}$. \end{lem}

\begin{lem}\label{lem:LPlem1} Let $r$ be a prime. For all $x\ge 3$, the number of integers $n\le x$ divisible by a prime $p\equiv 1\pmod{r}$ with 
  \[ \frac{r^2}{4 (\log{r})^2} < p \le r^2 (\log{r})^4 \]
is $\ll x(\log_2{(2r)})/r\log{r}$. 
\end{lem}

\begin{lem}[GRH-conditional]\label{lem:LPlem2} Suppose that $r$ is an odd prime and that $a$ is not an $r$th power. Let $A_{r}$ denote the set of primes $p\equiv 1\pmod{r}$ with $a^{(p-1)/r} \equiv 1\pmod{p}$. For $x\ge 3$, the number of $n\le x$ divisible by a prime $p\in A_{r}$ with $p \ge r^2 (\log{r})^4$ is $\ll x/r\log{r} + x\log_2{x}/r^2$. 
\end{lem}

Lemma \ref{lem:LPlem0} is the special case of \cite[Lemma 1]{LP03} where $\psi(x) = \log_2{x}$. Lemmas \ref{lem:LPlem1} and \ref{lem:LPlem2} are restatements of Lemmas 2 and 3 of \cite{LP03}, respectively.

\begin{proof}[Proof of Lemma \ref{lem:ordlem2} {\rm (}following the proof of Proposition 1 in \cite{LP03}{\rm )}] Suppose that $(P-1)/\ell_a(P)$ has a prime factor $r > w$. Then $P\equiv 1\pmod{r}$ and $\ell_a(P) \mid (P-1)/r$, so that $a^{(P-1)/r}\equiv 1\pmod{P}$. Hence, $P$ is a prime factor of $n$ belonging to the set $A_r$ from Lemma \ref{lem:LPlem2}. 

Thus, it is enough to show that only $o(x)$ positive integers $n\le x$ have a prime factor $p$ from any of the sets $A_r$, where $r>w$ is prime. We can assume by Lemma \ref{lem:LPlem0} that $\ell_a(p) > p^{1/2}/\log{p}$. Since $\ell_a(p) \le \frac{p-1}{r}$, we must have $p > \frac{r^2}{4(\log{r})^2}$. Noting that $a$ cannot be an $r$th power (once $x$ is large), Lemmas \ref{lem:LPlem1} and \ref{lem:LPlem2} put the remaining $n$ in a set of size 
\[ \ll x\sum_{\substack{r>w \\r\text{ prime}}} \left(\frac{\log_2{r}}{r\log{r}} + \frac{\log_2{x}}{r^2}\right) \ll x\frac{\log_4{x}}{\log_3{x}}, \]
which is $o(x)$.
\end{proof}

\section*{Acknowledgements} The author is supported by NSF award DMS-2001581.

\providecommand{\bysame}{\leavevmode\hbox to3em{\hrulefill}\thinspace}
\providecommand{\MR}{\relax\ifhmode\unskip\space\fi MR }
\providecommand{\MRhref}[2]{%
  \href{http://www.ams.org/mathscinet-getitem?mr=#1}{#2}
}
\providecommand{\href}[2]{#2}

\end{document}